\documentclass[12pt]{amsart}
\usepackage{amsfonts, amssymb, latexsym, hyperref, xcolor, tikz,tikz-cd,transparent}
\usepackage{ytableau}

\newtheorem{theorem}{Theorem}[section]
\newtheorem{proposition}[theorem]{Proposition}
\newtheorem{lemma}[theorem]{Lemma}

\newtheorem*{claim*}{Claim}
\newtheorem{corollary}[theorem]{Corollary}
\newtheorem{Main Conjecture}[theorem]{Main Conjecture}

\theoremstyle{definition}
\newtheorem{definition}[theorem]{Definition}
\theoremstyle{remark}

\newtheorem{example}[theorem]{Example}

\newtheorem{remark}[theorem]{Remark}
\theoremstyle{plain}

\newcommand\complexes{{\mathbb C}}
\newcommand\integers{{\mathbb Z}}

\newcommand\Groth{{\mathfrak G}}
\newcommand\groth{{\mathfrak G}}

\newcommand\reg{{\mathrm{reg}}}
\newcommand\post{{\mathrm{post}}}
\newcommand\init{{\mathrm{init}}}

\newcommand{\cellsize}{18}
\newlength{\cellsz} 
\newsavebox{\cell}
\sbox{\cell}{\begin{picture}(\cellsize,\cellsize)
\put(0,0){\line(1,0){\cellsize}}
\put(0,0){\line(0,1){\cellsize}}
\put(\cellsize,0){\line(0,1){\cellsize}}
\put(0,\cellsize){\line(1,0){\cellsize}}
\end{picture}}
\newcommand\cellify[1]{\def\thearg{#1}\def\nothing{}%
\ifx\thearg\nothing
\vrule width0pt height\cellsz depth0pt\else
\hbox to 0pt{\usebox{\cell} \hss}\fi%
\vbox to \cellsz{
\vss
\hbox to \cellsz{\hss$#1$\hss}
\vss}}
\newcommand\tableau[1]{\vtop{\let\\\cr
\baselineskip -16000pt \lineskiplimit 16000pt \lineskip 0pt
\ialign{&\cellify{##}\cr#1\crcr}}}

\newcommand{\excise}[1]{}

\begin{document}
\pagestyle{plain}
\title{Schubert determinantal ideals are Hilbertian}
\author{Ada Stelzer}
\author{Alexander Yong}
\address{Dept.~of Mathematics, U.~Illinois at Urbana-Champaign, Urbana, IL 61801, USA} 
\email{astelzer@illinois.edu, ayong@illinois.edu}
\date{June 28, 2023}

\begin{abstract}
Abhyankar defined an ideal to be \emph{Hilbertian} if its Hilbert polynomial coincides with its Hilbert function
for all nonnegative integers. In 1984, he proved that the ideal of $(r+1)$-order minors of a generic $p\times q$ matrix is Hilbertian. We give a different proof and
a generalization to the \emph{Schubert determinantal ideals} introduced by Fulton in 1992. Our proof reduces to a simple upper bound for the Castelnuovo-Mumford regularity of these ideals. We further indicate the pervasiveness of the Hilbertian property in Schubert geometry.
\end{abstract}

\maketitle

\vspace{-.1in}

\section{Introduction}

\subsection{History and motivation}\label{sec:motivation}
Fix an $r$-dimensional vector space
$V$ over ${\mathbb C}$. $GL(V)$ acts on the space $V^{\oplus p}\oplus (V^*)^{\oplus q}$ of $p$ vectors and $q$ covectors. Hence it acts on 
${\mathbb C}[V^{\oplus p}\oplus (V^*)^{\oplus q}]$. The \emph{first fundamental theorem of invariant theory
for $GL(V)$} states that the invariant ring ${\mathbb C}[V^{\oplus p}\oplus (V^*)^{\oplus q}]^{GL(V)}$ is generated
by \emph{contractions} $X_{ij}$ where 
$X_{ij}(v_1,\ldots,v_p; \phi_1,\ldots,\phi_q)=\phi_j(v_i)$.
The \emph{second fundamental theorem of invariant theory for $GL(V)$} gives the (first) syzygies between the 
contractions, i.e., it asserts a $\complexes$-algebra isomorphism 
\[R_{r,p,q}:={\mathbb C}[x_{ij}:1\leq i\leq p, 1\leq j\leq q]/I_{r,p,q} \cong {\mathbb C}[V^{\oplus p}\oplus (V^*)^{\oplus q}]^{GL(V)}\]
induced by the map $x_{ij}\mapsto X_{ij}$,
where $I_{r,p,q}$ is the ideal of $(r+1)\times (r+1)$ minors of a $p\times q$ matrix.\footnote{Better yet, one has a minimal free resolution of $R_{r,p,q}$; see work of Lascoux \cite{Las:res} and Weyman \cite{Weyman}.} In this way, the \emph{determinantal variety} ${\mathfrak X}_{r,p,q}$ 
defined by $I_{r,p,q}$ is connected to invariant theory. A vector space basis of ${\mathbb C}[V^{\oplus p}\oplus (V^*)^{\oplus q}]^{GL(V)}$ and hence of $R_{r,p,q}$ was given with
\emph{Young bitableaux} by Doubilet--Rota--Stein \cite{Doubilet}. This basis and its \emph{straightening law} were further explained  by De Concini--Procesi
\cite{DeConciniProcesi}, and used to study determinantal varieties \emph{per se} by 
De Concini--Eisenbud--Procesi \cite{DEP}. These determinantal varieties were shown to be open neighborhoods of certain Schubert varieties by Lakshmibai--Seshadri \cite{LS78} (see the survey \cite{WY} and the references therein). 

Abhyankar studied the Hilbert function of 
$R_{r,p,q}$. In \cite[Theorem 5]{Abh84} he gave a formula for the function and used it to prove  that $I_{r,p,q}$ is \emph{Hilbertian}, that is, the Hilbert function agrees with the
Hilbert polynomial for all nonnegative integers rather than merely in the long run. 
Abhyankar--Kulkarni \cite[Section~4, Main Theorem]{AK} gave a generalization of this result to \emph{ladder determinantal ideals}.
See Ghorpade's survey \cite{Ghorpade} on Abhyankar's work for further elaboration and references.

\subsection{Schubert determinantal ideals}
We give a different proof of Abhyankar's Hilbertian theorem \cite[Theorem~5]{Abh84}, together with a new generalization to \emph{matrix Schubert varieties}. This work complements the aforementioned Abhyankar-Kulkarni theorem \cite{AK} as well as recent work on the  regularity of matrix Schubert varieties due to Rajchgot--Ren--Robichaux--St.~Dizier--Weigandt \cite{Robichaux1}, Rajchgot--Robichaux--Weigandt \cite{Robichaux2} and Pechenik--Speyer--Weigandt \cite{PSW}.

Let ${\sf Mat}_{n}$ be the space of $n\times n$ matrices with entries in ${\mathbb C}$. Let $GL_n\subseteq {\sf Mat}_n$ be the group of invertible matrices with Borel subgroup $B$ of upper triangular matrices and opposite Borel $B_-$ of lower triangular matrices. Now, $B_{-}\times B$ acts on ${\sf Mat}_{n}$ by $(b_{-},b)\cdot M=b_{-}Mb^{-1}$. If $w$ is a permutation in the symmetric group $S_n$ on $[n]:=\{1,2,\ldots,n\}$, let $M_w$ be its permutation matrix with $1$'s in positions $(i,w(i))$ and $0$'s elsewhere.

\begin{definition}[\cite{Fulton:duke, Knutson.Miller}]
    The \emph{matrix Schubert variety} ${\mathfrak X}_w$ is the $B_{-}\times B$-orbit closure of $M_w$ in ${\sf Mat}_{n}$.\footnote{The nomenclature is justified as follows: ${\mathfrak X}_w$ is Zariski closure of $\pi^{-1}(X_w)$ in ${\sf Mat}_n$, where $\pi:GL_n\to GL_n/B$ is the natural projection to the flag variety $GL_n/B$ and $X_w\subseteq GL_n/B$ is a Schubert variety.} Its coordinate ring is denoted $R_w$.
\end{definition}

The \textit{Schubert determinantal ideal} $I_w\subseteq\complexes[{\sf Mat}_n]$ is the defining ideal of ${\mathfrak X}_w$ (see Section~\ref{section:strengthening}). 

In general, suppose $I\subset S={\mathbb C}[x_1,x_2,\ldots,x_N]$ is a homogeneous ideal. Then  $R:=S/I$ has the graded decomposition
\[R=\bigoplus_{k\geq 0} R_k.\]
The \emph{Hilbert function} is defined by $HF_R(k)=\dim_{\mathbb C}(R_k)$. For sufficiently large values of $k$ the values $HF_R(k)$ match those of a polynomial, called the \emph{Hilbert polynomial} $HP_R(k)$. 
\begin{definition}[\cite{Abh84}]
    $R = S/I$ is \emph{Hilbertian} if $HF_R(k)=HP_R(k)$ for all $k\in {\mathbb Z}_{\geq 0}$. If $HF_R(k)=HP_R(k)$ for all $k\geq 1$ then we say $R$ is \emph{almost Hilbertian}.\footnote{The name ``Hilbertian'' is credited to \cite{Abh84} in \cite{AK} but does not actually appear in the former article of Abhyankar prepared by Galligo.}
\end{definition}

The first version of our main theorem for Schubert determinantal ideals is as follows:

\begin{theorem}\label{thm:main}
    $R_w$ is Hilbertian for any $w\in S_n$.
\end{theorem} 

The question of when the Hilbert function and polynomial of $R = S/I$ begin to agree is answered by computing the \emph{(Castelnuovo-Mumford) regularity} $\reg(R)$. Recent work has established excellent combinatorial comprehension of this statistic for coordinate rings of matrix Schubert varieties. Previously, Knutson--Miller \cite{Knutson.Miller}
gave a formula for the Hilbert series of $R_w$;  the numerator is a \emph{Grothendieck polynomial} ${\mathfrak G}_w$. Using this,
\cite{Robichaux1}, Rajchgot--Ren--Robichaux--St.~Dizier--Weigandt made the fruitful observation that
\begin{equation}\label{eqn:RRRSW}
    \deg(\Groth_w)=\reg(R_w)+\mathrm{codim}(\mathfrak{X}_w);
\end{equation}
we use this idea in our proof. They gave a combinatorial rule for this regularity in the case where $w$ is Grassmannian (has a single descent). Rajchgot--Robichaux--Weigandt \cite{Robichaux2} generalized this to the case that $w$ is $2143$-avoiding or is $1423$-avoiding. Pechenik--Speyer--Weigandt \cite{PSW} gave a rule for regularity for general $w\in S_n$. 

As we shall explain, the regularity of $R_w$ (or rather, its \emph{postulation number}) is so small that its precise value is not needed for proving the Hilbertian property. Instead, we use 
a weak upper bound on the degree of a Grothendieck polynomial (Proposition~\ref{prop:thebound}). We give a short proof of this bound, \emph{ab initio}, from the ``graphical'' formulation \cite{Knutson.Yong} of Lascoux's transition formula for Grothendieck polynomials \cite{Lascoux:transition}. See also Remark~\ref{remark:simpler}. 

Strictly speaking, the family of varieties ${\mathfrak X}_{r,p,q}$ is \emph{not} a subfamily of the matrix Schubert varieties. Each ${\mathfrak X}_{r,p,q}$ is only equal to some ${\mathfrak X}_{w}$ up to a Cartesian product with affine space. While such Cartesian products do not change homological invariants, including regularity (see Proposition~\ref{prop:regstable}), they can affect the Hilbertian property. 

\begin{example}
     $R={\mathbb C}[x,y]/\langle x\rangle$ has $HF_{R}(k)=1$ for all $k$ while $R'={\mathbb C}[x]/\langle x\rangle$ has $HF_{R}(0)=1$ and $HF_R(k)=0$ for $k>0$. Thus the coordinate ring of $\langle x\rangle \subset {\mathbb C}[x,y]$ is Hilbertian while that of $\langle x\rangle \subset {\mathbb C}[x]$ is not.
\end{example}

The strengthening of Theorem~\ref{thm:main} given in Theorem~\ref{thm:strongermain}
shows that the Schubert determinantal ideals remain Hilbertian even after removing ``irrelevant variables''. This strengthened version generalizes Abhyankar's Hilbertian theorem (see Example~\ref{exa:eff}). Theorem~\ref{thm:strongermain} and Theorem~\ref{thm:main} are proved together in Section~\ref{sec:proofs}.

The 
first fundamental theorem of invariant theory for $GL(V)$ is equivalent to \emph{Schur-Weyl duality} (see \cite{GoodmanWallach}) and thereby connected to representation theory of general linear groups. We remark that just as $R_{r,p,q}$ is a $GL_p$-module, $R_w$ is a module for choices of (reductive) Levi subgroups $L\leq GL_n$ depending on the descent positions of $w$. The representation-theoretic decomposition of $R_w$ into $L$-irreducibles is of significance to the study of the
Hilbert function of $R_w$. We hope to address this matter in future work.

\subsection{The prevalence of the Hilbertian property}
As evidenced in this paper, many varieties related to Schubert geometry are Hilbertian. The simplicity of the arguments in this paper suggest similar proofs in other instances. Indeed, any Cohen--Macaulay square-free monomial ideal has an almost Hilbertian coordinate ring (Theorem~\ref{thm:SRKpoly}). As observed (using different language) by Bruns--Herzog in \cite{BHa-inv} and recorded in Corollary~\ref{cor:SRHilbertian} here, such a coordinate ring is Hilbertian if and only if the reduced Euler characteristic of the associated simplicial complex is $0$. In our experience, the Hilbert function and polynomial typically agree (i.e., equal $0$) for inputs well below $0$.

After the original version of this preprint was made public, 
we obtained feedback from Knutson that inspired us to characterize the Hilbertian \emph{subword complexes} of Knutson--Miller in Theorem~\ref{thm:subwordHilbertian}. Realizing $I_w$ and $\widetilde{I}_w$ via subword complexes yields a second proof of Theorems~\ref{thm:main} and \ref{thm:strongermain}.
Combining Theorem~\ref{thm:subwordHilbertian} with 
Knutson \cite{Knutson:patches} we can classify
Kazhdan-Lusztig ideals that are Hilbertian (when the ideals are standard graded); see Theorem~\ref{thm:KL}. We similarly classify the Hilbertian \textit{skew-symmetric matrix Schubert varieties} of Marberg-Pawlowski \cite{Marberg.Pawlowski}. We expect that our proofs of Theorem~\ref{thm:main}, combined with
work of Kinser-Rajchgot \cite{Kinser.Rajchgot} and/or Kinser-Knutson-Rajchgot \cite{Kinser} will
permit a characterization the Hilbertian \textit{quiver loci} for an $A_n$-quiver of arbitrary orientation. 
In addition, we conjecture the Hilbertian property for the tangent cones to a Schubert variety $X_w$ at a $T$-fixed point $e_v$ when $v<w$ in Bruhat order; see \cite{Yong}. We confirmed this conjecture by computer for $n\leq 6$ and can prove it with Theorem~\ref{thm:main} if $w$ is ``covexillary''.

\section{Preliminaries}

\subsection{Regularity and Hilbertian ideals}
We review standard commutative algebra that we need to state and prove Theorem~\ref{thm:main}, with \cite[Chapter 6]{CLO2}, \cite[Sections 8.2 and 8.3]{Miller.Sturmfels} and \cite[Chapter 4]{BH} as our references. Let $S = \complexes[x_1,\dots,x_N]$ be a standard graded polynomial ring. For any homogeneous ideal $I\subseteq S$, a \emph{(graded) free resolution} $F_\bullet$ of $S/I$ is a sequence $\{F_i\}_{i=0}^\infty$ of free $S$-modules $F_i$ connected by degree-0 graded maps $\{\partial_i\}_{i=1}^\infty$ as follows:
\begin{equation*}
    \cdots\xrightarrow{\partial_{k+1}} F_k\xrightarrow{\partial_k} F_{k-1}\xrightarrow{\partial_{k-1}}\cdots\xrightarrow{\partial_1} F_0\to S/I\to 0.
\end{equation*}
We require that this sequence be \emph{exact}, meaning that the image of each map is the kernel of the next. If $F_k\neq 0$ and $F_i = 0$ for all $i>k$, then $F_\bullet$ has \textit{length} $k$. Let $S(-j)$ denote a copy of $S$ with all degrees shifted up by $j$ (so $\deg x_i = 1+j$). Then each free module $F_i$ in a free resolution $F_\bullet$ can be uniquely expressed as $\bigoplus_{j\in\integers} S(-j)^{b_{ij}}$ for some non-negative integers $b_{ij}$. The maps $\partial_i$ in $F_\bullet$ can be written as matrices with entries in $S$; $F_\bullet$ is called  \textit{minimal} if none of these entries are units. Equivalently, $F_\bullet$ is minimal if it simultaneously minimizes the values of all $b_{ij}$ among free resolutions of $S/I$. Hilbert proved that $S/I$ always has a minimal free resolution of length at most $n$, which is unique up to isomorphism \cite[Theorems 6.3.8 and 6.3.13]{CLO2}. The values $b_{ij}$ occurring in the minimal free resolution of $S/I$ are called the \textit{(graded) Betti numbers} of $S/I$ and denoted $\beta_{ij}$.

\begin{definition}
    The \emph{(Castelnuovo-Mumford) regularity} of $S/I$ is
    \begin{equation*}
        \reg(S/I) = \max\{j-i|\beta_{ij}\neq 0\}.
    \end{equation*}
\end{definition}

We sometimes abuse notation by referring to the regularity of an ideal when we mean the regularity of its coordinate ring. This abuse is convenient because regularity is stable under inclusions of ideals into larger polynomial rings, as the next proposition shows.

\begin{proposition}\label{prop:regstable}
    Let $I\subseteq S = \complexes[x_1,\dots,x_N]$ be a homogeneous ideal, $T = \complexes[y_1,\dots, y_M]$, and $R = S\otimes_\complexes T = \complexes[x_1,\dots,x_N, y_1,\dots,y_M]$. Then $\reg(S/I) = \reg(R/I)$.
\end{proposition}
\begin{proof}
    Let $F_\bullet$ be the minimal free resolution of $S/I$. Let $G_\bullet = F_\bullet\otimes_\complexes T$, meaning $G_i = F_i\otimes_\complexes T$ and $\partial^G_i = \partial^F_i\otimes_\complexes id_T$ for all $i\geq 0$. Since tensor products distribute over direct sums, we can express $G_i$ as a direct sum of free $R$-modules $S(-j)\otimes_\complexes T\cong R(-j)$. Thus $G_i = \bigoplus_{j\in\integers} R(-j)^{\beta_{ij}}$ where the $\beta_{ij}$ are the graded Betti numbers of $S/I$. Furthermore, the functor $-\otimes_\complexes T$ is exact (indeed, tensoring over a field is always exact), so $G_\bullet$ forms an exact sequence and is therefore a resolution of $R/I$. This resolution is minimal because the matrices representing each $\partial^G_i$ are given by Kronecker products of the matrices representing $\partial^F_i$ and $id_T$, and the entries of $\partial^F_i$ are non-units by assumption. Thus the Betti numbers of $S/I$ and $R/I$ are the same, so in particular $\reg(S/I) = \reg(R/I)$ as claimed.
\end{proof}

\begin{remark}
    Let $I\subseteq S$ and $J\subseteq T$ be ideals, and let $F_\bullet$ and $G_\bullet$ be minimal free resolutions for $S/I$ and $T/J$ respectively. Then $F_\bullet\otimes_\complexes G_\bullet$ is always a minimal free resolution for $(S/I)\otimes_\complexes (T/J)\cong R/(I+J)$. Proposition~\ref{prop:regstable} is the special case where $J = (0)$.
\end{remark}

\begin{definition}
    For a homogeneous ideal $I\subseteq S$, the \emph{Hilbert function} of $S/I$ is the function $HF_{S/I}$ sending each non-negative integer $k$ to the dimension of the grade-$k$ component of $S/I$ (viewed as a $\complexes$-vector space).
\end{definition}

The \emph{Hilbert series} of $S/I$ is the formal generating series for $HF_{S/I}$. This series is a rational function whose numerator is called the \emph{K-polynomial}\footnote{In some sources the $K$-polynomial is simply called the \textit{Hilbert numerator}.} $K_{S/I}(t)$ \cite[Theorem 8.20]{Miller.Sturmfels}:
\begin{equation*}
    \sum_{k=0}^\infty HF_{S/I}(k)t^k = \frac{K_{S/I}(t)}{(1-t)^N}.
\end{equation*}
The Hilbert function agrees with a (unique) polynomial for sufficiently large input; this polynomial is the \emph{Hilbert polynomial} $HP_{S/I}$ \cite[Proposition 6.4.7]{CLO2}. The \emph{postulation number} of $S/I$ captures the ``sufficiently large" condition exactly:
\begin{equation*}
    \post(S/I) = \max\{k: HF_{S/I}(k)\neq HP_{S/I}(k)\}.
\end{equation*}
The coordinate ring $S/I$ is called \textit{Hilbertian} if $HF_{S/I}(k) = HP_{S/I}(k)$ for all $k\geq 0$, i.e., if $\post(S/I) < 0$. The postulation number $\post(S/I)$ and $K$-polynomial $K_{S/I}(t)$ are related via the regularity of $S/I$. This relationship is given by Lemma~\ref{lemma:postformula}: the first part is considered well-known by experts
and can be obtained from \cite[Theorem~4.4.3]{BH}, while the second is explicitly \cite[Proposition 4.1.12]{BH}.

\begin{lemma}[{\cite[Theorem~4.4.3, Proposition~4.1.12]{BH}}]\label{lemma:postformula}
    Let $I\subseteq S = \complexes[x_1,\dots,x_N]$ be a homogeneous ideal such that $S/I$ is Cohen-Macaulay. Let $X$ be the variety for $S/I$. Then
    \begin{enumerate}
    \item $\reg(S/I) = \deg(K_{S/I}(t)) - \mathrm{codim}(X)$.
    \item $\post(S/I) = \reg(S/I) - \dim(X) = \deg(K_{S/I}(t)) - N$.
    \end{enumerate}
\end{lemma}

Proving $S/I$ is Hilbertian therefore reduces to showing that $\deg(K_{S/I}(t)) < N$. Similarly, $S/I$ is almost Hilbertian when $\deg(K_{S/I}(t))\leq N$.

\begin{remark}
    Theorem 4.4.3(c) and the subsequent discussion in \cite{BH} show that $\post(S/I)$ agrees with the \textit{$a$-invariant} $a(S/I)$ when $S/I$ is Cohen--Macaulay. In \cite{BHa-inv}, Bruns--Herzog compute the $a$-invariants of classical determinantal ideals using the methods below.
\end{remark}

\subsection{Stanley-Reisner theory and Gr\"obner degeneration}
The combination of Stanley-Reisner theory and Gr\"obner degeneration provide a convenient method for computing Hilbert series. We review the basics with \cite{Miller.Sturmfels} as reference.

\begin{definition}
    Let $I\subseteq S = \complexes[x_1,\dots,x_N]$ be a square-free monomial ideal. The \emph{Stanley--Reisner complex} associated to $I$ is the simplicial complex $\Delta_I$ on $[N]$ whose faces are indexed by the square-free monomials not in $I$.
\end{definition}

The Stanley--Reisner construction gives a bijection between simplicial complexes on $[N]$ and square-free monomial ideals in $\complexes[x_1,\dots,x_N]$ \cite[Theorem 1.7]{Miller.Sturmfels}. When $I$ is a square-free monomial ideal, this allows one to read the $K$-polynomial of $S/I$ off from $\Delta_I$ directly.

\begin{theorem}[{\cite[Theorem 1.13]{Miller.Sturmfels}}]\label{thm:SRKpoly}
    Let $I\subseteq S$ be a square-free monomial ideal. Then $$K_{S/I}(t) = \sum_{\sigma\in\Delta_I}\left(\prod_{i\in\sigma}t\cdot\prod_{j\notin\sigma}(1-t)\right).$$
    In particular, $\deg(K_{S/I}(t))\leq N$ (so if $S/I$ is Cohen--Macaulay it is almost Hilbertian).
\end{theorem}

\begin{corollary}[{\cite[pg. 207]{BHa-inv}}]\label{cor:SRHilbertian}
    If $I\subseteq S$ is a Cohen--Macaulay, square-free monomial ideal, then $S/I$ is Hilbertian if and only if the \textit{reduced Euler characteristic} $\sum_{\sigma\in\Delta_I}(-1)^{|\sigma|}$ is 0. 
\end{corollary}
\begin{proof}
    By Lemma~\ref{lemma:postformula} $S/I$ is Hilbertian if and only if $K_{S/I}(t)$ has degree strictly less than $N$. In the explicit formula for $K_{S/I}(t)$ of Theorem~\ref{thm:SRKpoly}, each face $\sigma\in \Delta_I$ contributes a factor of $(-1)^{|\sigma^c|}t^N$, where $|\sigma^c|$ is the number of vertices not in $\sigma$. Thus $S/I$ is Hilbertian if and only if all these $\pm t^N$ factors cancel, i.e., if $\sum_{\sigma\in\Delta_I}(-1)^{|\sigma^c|} = \pm\sum_{\sigma\in\Delta_I}(-1)^{|\sigma|} = 0$.
\end{proof}

\begin{remark}
    In the case where $\Delta_I$ is \textit{shellable}, \cite[Proposition 2.1]{BHa-inv} gives an explicit formula for the $a$-invariant/postulation number of $S/I$. This could be used to give a formula for $\reg(S/I)$ in
    the cases of coordinate rings of matrix Schubert varieties and 
    Kazhdan-Lusztig varieties \cite{WYKL} (that admit a dilation action, and in general Lie type, such as in \cite{Graham.Kreiman}) after Gr\"obner degeneration; see
    Section~\ref{sec:subword}. We do not pursue this point here.
\end{remark}

When $I$ is only assumed to be homogeneous, \textit{Gr\"obner degeneration} provides a standard procedure for reducing to the monomial case. The \textit{initial ideal} of $I$ with respect to a chosen term order $<$ is the monomial ideal $\init_<(I) = \langle LT_<(f)|f\in I\rangle$, where $LT_<(f)$ is the lead term of $f$ under $<$. The Hilbert functions of $I$ and $\init_<(I)$ are the same (cf. \cite[pg. 158]{Miller.Sturmfels}) and we immediately obtain the following result:

\begin{corollary}\label{cor:sqfreeHilbertian}
    If $I\subseteq S$ is an ideal such that $S/I$ is Cohen--Macaulay and $\init_<(I)$ is square-free for some choice of $<$, then $S/I$ is almost Hilbertian.
\end{corollary}

\begin{remark}
    In many situations one starts with an ideal $I$ and wishes to identify a term order $<$ such that $\init_<(I)$ is square-free. The contrapositive of Corollary~\ref{cor:sqfreeHilbertian} says that if $S/I$ is Cohen--Macaulay and \textit{not} almost Hilbertian, then no such term order exists. A simple example is given by the homogenized cuspidal cubic $\{x^3-y^2z = 0\}$ in $\complexes^3$.
\end{remark}

\begin{remark}
    The method for computing $K_{S/I}(t)$ above can be modified to work when $\init_<(I)$ is not square-free. \textit{Polarizing} $\init_<(I)$ yields a square-free monomial ideal $J$ such that $K_{S/J}(t) = K_{S/I}(t)$, to which Theorem~\ref{thm:SRKpoly} can be applied.\footnote{See \cite{SY:CCAR} which applies this to give an algorithm that constructs sets whose cardinality equals the degree of a homogeneous ideal $I$.} Polarization does not preserve the postulation number: in general we have $\post(S/J) \leq \post(S/I)$. 
\end{remark}

\subsection{Permutation combinatorics}
We need some standard permutation combinatorics; our reference is
\cite{Manivel}. The \textit{Coxeter length} of $w\in S_n$ is the number of inversions in $w$, i.e.,
$    \ell(w):=\#\{i<j:w(i)>w(j)\}$. 
The \emph{graph} of $w\in S_n$ places a $\bullet$ in each position $(i,w(i))$ (written in matrix notation). The \textit{Rothe diagram} of $w$, denoted $D(w)$, consists of all boxes in $[n]\times[n]$ not weakly below or right of a $\bullet$. We have
\[D(w)=\{(i,j)\in [n]\times [n]: j<w(i), i<w^{-1}(j)\}.\]
The \emph{essential set} $E(w)$ of $w$ is comprised of the maximally southeast boxes of each connected component of $D(w)$, i.e.,
\[E(w)=\{(i,j)\in D(w): (i,j+1), (i+1,j)\not\in D(w)\}.\]

\begin{definition}
    The \emph{effective region} of $w\in S_n$, denoted $\lambda(w)$, consists of $(i,j)\in [n]\times [n]$ such that
    $(i,j)$ is weakly northwest of some $(i',j')\in E(w)$.
\end{definition}

It follows immediately that $\lambda(w)$ has the shape of a Young diagram. 

\begin{definition}
    $w$ is \emph{dominant} if $\lambda(w)=D(w)$.
\end{definition}

It is convenient to work in \[S_{\infty}=\bigcup_{n\geq 1} S_n\] where we identify two permutations $w\in S_n ,w'\in S_n'$ for $n<n'$
if $w(i)=w'(i)$ for $1\leq i\leq n$ and $w'(i)=i$ for $n+1<i\leq n'$. In such a case, $D(w)\subseteq [n]\times [n]$ and $D(w')\subseteq [n']\times [n']$ have the same elements. Identifying these diagrams allows us to unambiguously refer to \emph{the} diagram $D(w)$ of $w \in S_{\infty}$.

Let $t_{a\leftrightarrow b}$ be the transposition on $S_{\infty}$ interchanging $a$ and $b$. Hence $wt_{a\leftrightarrow b}$ swaps the \emph{positions} $a$ and $b$. The simple transposition $s_i$ is $t_{i\leftrightarrow i+1}$. A \emph{descent} of a permutation $w$ is an index $i$ such that $\ell(ws_i)<\ell(w)$. 

\subsection{Grothendieck polynomials}
We recall the notion of
\emph{Grothendieck polynomial} due to Lascoux--Sch\"utzenberger \cite{LS:Hopf}. 
The definition we use is not the original one and is due to Lascoux \cite{Lascoux:transition} (see also \cite{Lenart:transition}). 
Let ${\bf x}=\{x_1,x_2,\ldots \ \}$ be a collection of
commuting independent variables.
For each $w \in S_\infty$,
there is a Grothendieck polynomial $\Groth_w({\bf x})$. These polynomials
satisfying the following recursion:

\begin{theorem}[Lascoux's Transition formula for Grothendieck polynomials \cite{Lascoux:transition}, cf.~\cite{Lenart:transition}]
\label{theorem:trans}
    Let $w\in S_{\infty}$ have last descent $g$,
    let $m>g$ be the largest integer such that
    $w(m)<w(g)$ and set $w'=w t_{g\leftrightarrow m}$.
    Suppose that $1\leq i_1<i_2<\ldots<i_s<g$ are the indices such that
    $\ell(w' t_{i_j\leftrightarrow g})=\ell(w')+1$. Then:
    \begin{equation}\label{eqn:K-transition}
        \groth_{w}(X) = \groth_{w'}(X)
        +(x_g -1)\big[\groth_{w'}(X)\cdot ({\sf Id}-t_{i_1\leftrightarrow g})\cdots
        ({\sf Id}-t_{i_s\leftrightarrow g})\big],
    \end{equation}
    where $t_{j\leftrightarrow l}$ acts on the $\{\groth_{u}(X)\}$ by
    $\groth_{u}(X)\cdot t_{j\leftrightarrow l}
    = \groth_{u t_{j\leftrightarrow l}}(X)$
    and ${\sf Id}$ is the identity operator.
\end{theorem}

Theorem~\ref{theorem:trans} uniquely determines all $\Groth_w$ from the base case $\groth_{id} = 1$.

\section{A strengthened version of Theorem~\ref{thm:main}}\label{section:strengthening}
Recall from the introduction that the \textit{matrix Schubert variety} $\mathfrak{X}_w$ is the $B_-\times B$-orbit closure of the permutation matrix $M_w$ in $\sf{Mat}_n$. It is an affine variety of codimension $\ell(w)$. The defining ideal $I_w$ of $\mathfrak{X}_w$ is called the \emph{Schubert determinantal ideal}. Make the natural
identification ${\mathbb C}[{\sf Mat}_n]={\mathbb C}[z_{ij}:1\leq i,j\leq n]$ where $z_{ij}$ is the $(i,j)$-coordinate function. Fulton \cite{Fulton:duke} produced generators for $I_w$ as follows. Let $r_{ij}$ count the number of $1$'s in the northwest $i\times j$ submatrix of $M_w$. Let $Z=(z_{ij})_{1\leq i,j\leq n}$ be the generic $n\times n$ matrix and set $Z_{ij}$ to be the northwest $i\times j$ submatrix of $Z$. Then
\begin{equation}\label{eqn:thegenerators}
    I_w=\langle \text{rank $r_{ij}+1$ minors of $Z_{ij}$, $1\leq i,j\leq n$}\rangle.
\end{equation}
Moreover in \emph{ibid.} it is proved that $I_w$ is a prime ideal and the coordinate ring
\[R_w:={\mathbb C}[{\mathfrak X}_w]={\mathbb C}[{\sf Mat}_n]/I_w\]
is Cohen-Macaulay. By Lemma \ref{lemma:postformula}(2) it follows that
\begin{equation*}
    \post(R_w) = \deg(K_{R_w}(t))-n^2.
\end{equation*}
The $K$-polynomial $K_{R_w}(t)$ is known. By \cite{Knutson.Miller} it is a Grothendieck polynomial.
\begin{equation*}
    K_{R_w}(t) = \Groth_w(x_i\mapsto 1-t).
\end{equation*}
Theorem~\ref{thm:main} therefore reduces to the statement that $\deg(\Groth_w)< n^2$ for all $n\geq 1$ and $w\in S_n$.

Fulton \cite{Fulton:duke} refined the list of generators for $I_w$ by showing that
\begin{equation}\label{eqn:smallerlist}
    I_w=\langle \text{rank $r_{ij}+1$ minors of $Z_{ij}$, where $(i,j)\in E(w)$}\rangle.\footnote{For a minimal list of generators see \cite{GaoYong}.}
\end{equation}

Notice that the generators of $I_w$ in (\ref{eqn:smallerlist}) only involve the variables $z_{ij}$ where
$(i,j)\in \lambda(w)$. It therefore makes sense to think about the determinantal variety that only uses these
``effective'' variables. Formally, let us define the \emph{effective Schubert determinantal ideal} 
\[{\tilde I}_w\subset {\mathbb C}[z_{ij}:(i,j)\in \lambda(w)]\] 
where ${\tilde I}_w$ uses the same generators
as in (\ref{eqn:smallerlist}). Thus the \emph{effective matrix Schubert variety} ${\widetilde X}_w$ is the zero-locus of these equations
inside the affine space ${\mathbb C}^{|\lambda(w)|}$ rather than ${\mathbb C}^{n^2}$. Let ${\widetilde R}_w$ denote the coordinate ring of ${\widetilde X}_w$. One has a trivial isomorphism,
${\widetilde {\mathfrak X}}_w \times {\mathbb C}^{n^2-|\lambda(w)|} \cong {\mathfrak X}_w$.

\begin{example}\label{exa:eff}
    The effective Schubert determinantal ideals ${\widetilde I}_w$ generalize  $I_{r,p,q}$. Let 
    $w_{r,p,q}=1 \ 2 \ \ldots \ r \ q+1 \ q+2 \ \ldots \ q+p \ r+1 \ r+2 \ \ldots q$.
    Then $\lambda(w_{r,p,q})$ is a $p\times q$ rectangle and $r_{pq} = r$, so the ambient ring of 
    ${\widetilde I}_{w_{r,p,q}}$ is ${\mathbb C}[z_{ij}: 1\leq i\leq p, 1\leq j\leq q]$ and
    ${\widetilde I}_{w_{r,p,q}}=I_{r,p,q}$.\footnote{$w_{r,p,q}$ is a well-known construction. A
    \emph{bigrassmannian permutation} $w\in S_\infty$ is one where $w$ and $w^{-1}$ each have at most one descent.
    All bigrassmannian permutations are of the form $w_{r,p,q}$ for choices of the parameters $r,p,q$.}
 \end{example}

We are now ready to state our strengthened version of Theorem~\ref{thm:main}:

\begin{theorem}\label{thm:strongermain}
    ${\widetilde I}_w$ is Hilbertian if and only if $w$ is not dominant. ${\widetilde I}_w$ is almost Hilbertian for all $w$. 
\end{theorem}

\begin{remark}
    If $w$ is dominant, then $I_w= \langle z_{ij}: (i,j)\in \lambda(w)\rangle$
    is generated by $1\times 1$ minors. Hence ${\widetilde X}_w=\{{\bf 0}\}$. Theorem~\ref{thm:strongermain} therefore asserts that ${\widetilde R}_w$ is Hilbertian unless ${\widetilde R}_w\cong\complexes$.
\end{remark}

\begin{example}
    The Rothe diagram for $w = 24315$ is presented below, with the values of $r_{ij}$ displayed only in elements of $E(w)$. The effective region $\lambda(w)$ is outlined in red.
    
    
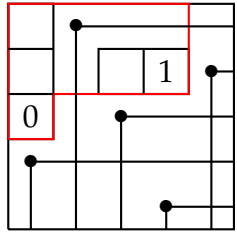
\begin{figure}[h]
        \centering
        \begin{tikzpicture}[scale = 0.30]
           \draw[black, thick] (0,0) -- (10,0) -- (10,10) -- (0,10) -- (0,0);
            \draw[black, thick] (0, 4) -- (2, 4) -- (2, 10) -- (0,10) -- (0,4);
            \draw[black, thick] (4, 6) -- (8, 6) -- (8, 8) -- (4, 8) -- (4, 6);
            \draw[red, thick] (0, 4) -- (2, 4) -- (2, 6) -- (8, 6) -- (8, 10) -- (0, 10) -- (0, 4);
            \draw[black, thick] (1, 0) -- (1, 3) -- (10, 3);
            \draw[black, thick] (3, 0) -- (3, 9) -- (10, 9);
            \draw[black, thick] (5, 0) -- (5, 5) -- (10, 5);
            \draw[black, thick] (7, 0) -- (7, 1) -- (10, 1);
            \draw[black, thick] (9, 0) -- (9, 7) -- (10, 7);
            \draw[black, thick] (0, 6) -- (2, 6);
            \draw[black, thick] (0, 8) -- (2, 8);
            \draw[black, thick] (6, 6) -- (6, 8);
            \draw (1, 3) node{$\bullet$};
            \draw (3, 9) node{$\bullet$};
            \draw (5, 5) node{$\bullet$};
            \draw (7, 1) node{$\bullet$};
            \draw (9, 7) node{$\bullet$};
            \draw (1, 5) node{0};
            \draw (7, 7) node{1};
        \end{tikzpicture}
    \end{figure}

    The Fulton generators of $I_w$ are the variables $x_{11}$, $x_{21}$, $x_{31}$, along with the six $2\times 2$ minors of the $2\times 4$ matrix $\left(\begin{smallmatrix} x_{11} & x_{12} & x_{13} & x_{14}\\ x_{21} & x_{22} & x_{23} & x_{24}\end{smallmatrix}\right)$. Then $I_w$ is the ideal with these generators in $\complexes[x_{ij}|1\leq i, j\leq 5]$, while $\widetilde{I}_w$ lies in the subring generated by the 9 variables corresponding to boxes in $\lambda(w)$. One can verify that $\deg(\Groth_w) = 6$, so by Lemma~\ref{lemma:postformula}(2) we have $\post(R_w) = 6-25 = -17$ and $\post(\widetilde{R}_w) = 6-9 = -3$. Thus both ideals are Hilbertian in accordance with Theorems \ref{thm:main} and \ref{thm:strongermain}.
\end{example}

Notice that by Proposition~\ref{prop:regstable} we have $\reg({\widetilde R}_w) = \reg(R_w)$. It follows from Lemma~\ref{lemma:postformula}(2) that $\post({\widetilde R}_w) = \deg(\Groth_w) - |\lambda(w)|$. In particular, Theorem~\ref{thm:strongermain} reduces to the claim that $\deg(\Groth_w)\leq |\lambda(w)|$ for all $w$, with equality if and only if $w$ is a dominant partition.

\section{First proof of Theorems~\ref{thm:main} and~\ref{thm:strongermain}} \label{sec:proofs}

\subsection{A (simple) degree bound for Grothendieck polynomials}

\begin{proposition}\label{prop:thebound}
    $\deg(\Groth_w)\leq |\lambda(w)|$, with equality if and only if $w$ is dominant.
\end{proposition}

This bound is quite weak. For example, $\deg({\mathfrak G}_{s_k})=k$
whereas $\lambda(s_k)=k^2$.
Nonetheless, it suffices for our needs.

\noindent
\emph{Proof of Proposition~\ref{prop:thebound}:}
    To prove the inequality, it helps to think about what (\ref{eqn:K-transition}) means in terms of $D(w)$. This is explained in the ``diagram moves'' description of transition found in \cite[Section~2]{Knutson.Yong}. We refer the reader to that paper for the straightforward translation. The point is that in the transition formula, $(g, w(m))$ is a maximally southeast element of $E(w)$ and $D(w') = D(w)\setminus\{(g, w(m))\}$, so $\lambda(w')\subsetneq\lambda(w)$. The Grothendieck polynomials $\Groth_{w''}$ appearing in  $\groth_{w'}(X)\cdot ({\sf Id}-t_{i_1\leftrightarrow g})\cdots
    ({\sf Id}-t_{i_s\leftrightarrow g})$ all satisfy $\lambda(w'')\subseteq\lambda(w')$, implying
    \begin{equation}\label{eqn:cde}
        \lambda(w'')\subsetneq \lambda(w).
    \end{equation}
    Therefore the desired inequality follows from Theorem~\ref{theorem:trans} by induction on $|\lambda(w)|$.

    Now suppose that $w$ is not dominant but $w'$ is dominant. Then $(g, w(m))$ comprises an entire connected component of $D(w)$ and it follows that $\lambda(w')$ has at least two fewer elements than $\lambda(w)$ (for example, $\lambda(w')$ cannot contain $(g-1, w(m))$ or $(g, w(m)-1)$). The same induction as above then shows that $\deg(\Groth_w) < |\lambda(w)|$ when $w$ is not dominant.

    Conversely, suppose that $w$ is dominant. Then $w'$ is also dominant and (\ref{eqn:K-transition}) reduces to
     $\Groth_{w}({\bf x})=x_g \Groth_{w'}$. By induction, 
    $\Groth_{w}({\bf x})=\prod_{i: (i,j)\in \lambda(w)} x_i$
    is a monomial, of degree $|\lambda(w')|$. This establishes  ``$\Leftarrow$'' of the equality characterization and completes the proof.\qed

\subsection{Conclusion of the proofs of Theorems~\ref{thm:main} and~\ref{thm:strongermain}}
As noted in Section~\ref{section:strengthening}, Theorem~\ref{thm:main} reduces to showing that $\deg(\Groth_w) < n^2$. Theorem~\ref{thm:strongermain} similarly reduces to showing the stronger inequality $\deg(\Groth_w) \leq |\lambda(w)|$ with equality if and only if $w$ is dominant. Thus Proposition~\ref{prop:thebound} proves both theorems. \qed
 
\begin{remark}[A simpler degree bound]\label{remark:simpler}
    It is well-known, and trivial to prove using the combinatorial formula for Grothendieck polynomials from \cite{Fomin.Kirillov}, that $\deg(\Groth_w)\leq {n\choose 2}$. This upper bound on $\reg(R_w)$ is independent from $w$ and gives an  easier proof of Theorem~\ref{thm:main} since ${n\choose 2} < n^2$ for $n>0$. The weaker bound is not sufficient to prove Theorem~\ref{thm:strongermain}.
\end{remark}

\section{Hilbertian subword complexes, Kazhdan-Lusztig ideals, and more}\label{sec:subword}
In this section we generalize Theorem~\ref{thm:main} to the subword complexes of Knutson--Miller \cite{Knutson.Miller, Knutson.Miller:subword}. We sketch a second proof of Theorem~\ref{thm:strongermain} from this perspective before proving an analogous result for Kazhdan--Lusztig varieties. Finally, we comment on Hilbertian type-$A$ quiver loci and skew-symmetric matrix Schubert varieties.

\begin{definition}
    Let $\Pi$ be a Coxeter group and $\Sigma$ a minimal set of simple reflections generating $\Pi$. Let $Q = (\sigma_1,\dots,\sigma_m)$ be a sequence of elements in $\Sigma$ and let $\pi\in\Pi$. The \textit{subword complex} $\Delta(Q, \pi)$ is the simplicial complex on $[m]$ whose facets are subsequences of $Q$ whose complements are reduced expressions for $\pi$.
\end{definition}

Subword complexes are shellable by \cite[Theorem 2.5]{Knutson.Miller:subword}. Thus the associated coordinate rings $R_\Delta$ are Cohen--Macaulay (\cite[Theorem 13.45]{Miller.Sturmfels}) and we can check the Hilbertian property using Lemma~\ref{lemma:postformula}. 

\begin{theorem}\label{thm:subwordHilbertian}
    The Stanley--Reisner ring $R_\Delta$ associated to the subword complex $\Delta = \Delta(Q, \pi)$ is Hilbertian if and only if the Demazure product $\delta(Q)$ is equal to $\pi$.
\end{theorem}
\begin{proof}    
    Corollary~3.8 of \cite{Knutson.Miller:subword} shows that $\Delta(Q, \pi)$ is homeomorphic to a sphere if $\delta(Q) = \pi$ and a ball otherwise. The reduced Euler characteristic of a sphere is $1$, while the reduced Euler characteristic of a ball is $0$. Thus the result follows by Corollary~\ref{cor:SRHilbertian}.
\end{proof}

\noindent\emph{Second proof of Theorem~\ref{thm:strongermain} (sketch):}
The Grothendieck polynomial $\Groth_w$ arises as the $K$-polynomial of $\Delta = \Delta(Q, w)$ for an appropriate word $Q$ as seen in \cite[Section~1.8]{Knutson.Miller}. For different choices of $Q$ this complex corresponds to a Gr\"obner degeneration of $I_w$ or $\widetilde{I}_w$ (\cite[Theorem B]{Knutson.Miller}), so their coordinate rings are almost Hilbertian by Corollary~\ref{cor:sqfreeHilbertian}. It is immediately obvious from writing out these choices that $\delta(Q) = w$ if and only if $\Delta$ corresponds to $\widetilde{I}_w$ and $w$ is dominant. In this case, $\Delta$ is the empty complex, so the associated ring $R_\Delta$ is isomorphic to the base field $\complexes$. By Theorem~\ref{thm:subwordHilbertian} the result follows. \qed

The analogous result holds for \emph{Kazhdan-Lusztig varieties} with essentially the same proof. We refer the
reader to \cite{WYKL, Knutson:patches} for definitions of these varieties.

\begin{theorem}\label{thm:KL}
    Let $\mathcal{N}_{v, w}$ be a Kazhdan--Lusztig variety (in general Lie type) defined by homogeneous equations.\footnote{Kazhdan-Lusztig varieties are not given by (standard graded) homogeneous equations in general. In type $A$ it suffices to assume $v$ avoids $321$, see \cite[pg. 25]{Knutson:splitting}. Further analysis of which Kazhdan-Lusztig ideals are standard homogeneous is found in \cite[Section~5.1]{WYKL} but a complete classification is open. For examples in other Lie types, see \cite{Graham.Kreiman}.} Then the coordinate ring $R_{v, w}$ of $\mathcal{N}_{v, w}$ is Hilbertian unless $v = w$ (where $R_{v, w}\cong\complexes$).
\end{theorem}
\begin{proof}
    Kazhdan--Lusztig varieties $\mathcal{N}_{v, w}$ degenerate into varieties corresponding to subword complexes $\Delta = \Delta(Q, v)$, where $Q$ is a reduced word for $w$ by \cite[Theorem 4]{Knutson:patches}.\footnote{\cite[Proposition 3.3]{WYKL} gives an explicit construction of $\Delta(Q, v)$ in type $A$.} Thus $R_{v, w}\cong R_\Delta$ is Cohen--Macaulay, so it is Hilbertian if and only if $\delta(Q) = v$ by Theorem~\ref{thm:subwordHilbertian}. But $\delta(Q) = w$ since $Q$ is a reduced word for $w$, so
    $R_{v, w}$ is Hilbertian if and only if $w = v$. In this case $\Delta$ is the empty complex by definition, so $R_{v, w}\cong\complexes$ as claimed.
\end{proof}

\begin{remark}
    In \cite[Theorem~2.10]{Graham.Kreiman}, Graham--Kreiman give formulas for the Hilbert functions of local patches of Schubert varieties in a cominuscule generalized flag variety, thereby showing they are Hilbertian. Here is an alternative proof:  the varieties considered are cut out by homogeneous equations and thus fall under the purview of Theorem~\ref{thm:KL}.
\end{remark}

\begin{remark}
We expect that if $\Omega$ is a type-$A$ quiver locus with arbitrary orientation, then the coordinate ring $R_\Omega$ of $\Omega$ is Hilbertian unless $R_\Omega\cong \complexes$. Actually, we envision two possible proofs. The first 
is to use Theorem~\ref{thm:KL} combined with theorems of Kinser-Rajchgot \cite{Kinser.Rajchgot} that relate the quiver loci to Kazhdan-Lusztig ideals. The second is to use the Hilbert series formula for the loci from Kinser-Knutson-Rajchgot \cite{Kinser} combined with an analogue of our
first proof of Theorem~\ref{thm:main}. Both require enough notation and details to treat properly that we may do so elsewhere.
\end{remark}

\begin{remark}
    Marberg-Pawlowski carry out a study of \textit{skew-symmetric matrix Schubert varieties} in \cite{Marberg.Pawlowski}, developing a theory closely analogous to that of subword complexes for ordinary matrix Schubert varieties. Theorem~4.22 and Lemma~4.27 of their paper show that the coordinate rings of these varieties are isomorphic to $R_\Delta$ for a simplicial complex $\Delta$ homeomorphic to a ball or sphere (with a simple criterion to distinguish these cases). By Corollary~\ref{cor:SRHilbertian}, the Hilbertian rings are those where $\Delta$ is homeomorphic to a ball.
\end{remark}

\section*{Acknowledgements}
We thank Allen Knutson for suggestions that stimulated Section~\ref{sec:subword} after our original preprint was posted. We also thank Jenna Rajchgot and Reuven Hodges for helpful communications.
AS was partially supported by a Susan C.~Morosato IGL graduate student scholarship. 
AY was partially supported by a Simons Collaboration grant. This work was made possible in part by an NSF RTG in Combinatorics (DMS 1937241).

\end{document}